\documentclass [12pt,reqno]{amsart}
\usepackage{amsmath}
\usepackage{amsthm}
\usepackage{amsfonts, amssymb}
\usepackage[utf8]{inputenc} 
\usepackage[english]{babel}
\usepackage{cite}

\newtheorem{theorem}{Theorem}
\newtheorem*{theorem*}{Theorem}

\newtheorem{lemma}[theorem]{Lemma}

\theoremstyle{definition}

\theoremstyle{definition}

\theoremstyle{definition}

\def\R{{\mathbb R}}

\def\Chi{{\mathcal X}}
\def\Tau{{\mathcal T}}
\def\Nu{{\mathcal V}}

\def\<{\langle}
\def\>{\rangle}

\begin{document}

\title{Extensions of isometric embeddings of Pseudo-Euclidean metric polyhedra}%

\begin{abstract}
  We extend the results of B.~Minemyer by showing that any indefinite metric polyhedron (either compact or not) with the vertex degree bounded from above admits an isometric simplicial embedding into a Minkowski space of the lowest possible dimension. We provide a simple algorithm of constructing such embeddings. We also show that every partial simplicial isometric embedding of such space in general position extends to a simplicial isometric embedding of the whole space.
\end{abstract}

\author{Pavel Galashin}
\thanks{The first author was partly supported by the Chebyshev Laboratory  (Department of Mathematics and Mechanics, St. Petersburg State University)  under RF Government grant 11.G34.31.0026 and by JSC "Gazprom Neft"}
\address[Pavel Galashin]{Department of Mathematics, Massachusetts Institute of Technology, Cambridge, MA, 02139, United States}
\email[Pavel Galashin]{galashin@mit.edu}

\author{Vladimir Zolotov}
\address[Vladimir Zolotov]{Mathematics and Mechanics Faculty, St. Petersburg State University, Universitetsky pr., 28, Stary Peterhof, 198504, Russia.}
\email[Vladimir Zolotov]{paranuel@mail.ru}

\maketitle

\section{Introduction}

%
%
%

The question about isometrical simplicial embeddings of indefinite metric polyhedra into a Minkowski space was recently considered in \cite{M}. We show that the results from \cite{M} hold for non-compact indefinite metric polyhedra as well, and give an explicit construction.  We also show that every partial simplicial isometric embedding of indefinite metric polyhedra into a Minkowski space such that the images of the vertices are in d-general position extends to a simplicial isometric embedding of the whole space.


\subsection{Definitions}
An \textit{indefinite metric polyhedron} $\Chi$ is a simplicial complex endowed with a bilinear form attached to every simplex. The forms must agree on the intersections of the simplices, and are not assumed to be positive-definite or even non-degenerate.
For any edge ($1$-simplex) of this complex the value of the bilinear form on this edge is called its \textit{squared length}. The squared lengths of the edges can be arbitrary (including negative) real numbers, and these numbers completely determine the bilinear forms.


%
%
%

A \textit{simplicial map} of an indefinite metric polyhedron into a vector space $\R^n$ is a map that is affine on every simplex. 
Every simplicial map is completely determined by the images of the vertices.

A \textit{simplicial isometric map} of an indefinite metric polyhedron into 
a Minkowski space $\R^p_q$ with signature 
$(\underbrace{+,\dots,+}_p,\underbrace{-,\dots,-}_q)$ is a simplicial map such that the bilinear form of every simplex equals to
the pull-back of the inner product of $\R^p_q$ along this map. In particular, 
such maps preserve the squared lengths of edges. Conversely, every simplicial map that preserves squared lengths of edges is 
isometric, because a bilinear form is completely determined by its 
values on the edges of a non-degenerate simplex.



A \textit{simplicial isometric embedding} of a metric polyhedron into a Minkowski space is an injective simplicial isometric map.

For an indefinite metric polyhedron, by $\Nu$ we denote its vertex set. For 
every vertex $v$ of this set we define its \textit{degree} $\deg(v)$ as the 
number of edges of the simplicial complex containing $v$.

In the sequel we restrict our attention only to finite or countable indefinite 
metric polyhedra.


The special case of simplicial isometric map of Euclidean polyhedra into a Euclidean space has been studied in details in \cite{Krat, BZ,Bre,ak}. Note that in general a Euclidean polyhedron does not admit a simplicial isometric map into a Euclidean space. However, every polyhedron can be subdivided in such a way that the resulting polyhedron admits such a map, see, for example, \cite{Zalg}. In contrast, in the pseudo-Euclidean setting one does not need to subdivide the triangulation. In particular, a Euclidean polyhedron (of bounded vertex degree) admits a simplicial isometric embedding into a (non-Euclidean) Minkowski space of an appropriate dimension.   

Now we are ready to cite a theorem from \cite{M}:
\begin{theorem*}[B.~Minemyer, 2012, \cite{M}, Theorem~1.1, Corollary~3.4]
Let $\Chi$ be a compact $n$-dimensional indefinite metric polyhedron with vertex
set $\Nu$. Then
\begin{enumerate}
  \item There exists a simplicial isometric map of $\Chi$ into $\R^d_d$, where $d = \max\{\deg(v)|v \in \Nu\}$.
  \item There exists a simplicial isometric embedding of $\Chi$ into $\R^q_q$, where $q = \max\{d ,2n+1\}$. 
\end{enumerate}
\end{theorem*}

The dimension of $\R^d_d$ in the first part of the theorem is shown in \cite{M} to be optimal. The idea is that the $1$-skeleton of the standard $d$-dimensional simplex cannot be isometrically embedded into $\R^p_q$ with $p<d$.

This theorem is followed in \cite{M} by two generalizations. One of them 
provides an algorithm to explicitly construct such extensions, and the other one 
applies to non-compact polyhedra. Both of these upgrades result in a significant 
increase of the target space dimension. For non-compact polyhedra, the target 
space is $\R^p_p$, where $p = 2q(d^3-d^2+d+1)$, and $d$ and $q$ are the same as in the theorem.

In this paper we give an elementary proof of an extension theorem for such isometric embeddings of indefinite metric polyhedra. As a corollary, we get completely constructive versions of the theorems from \cite{M} for both compact and non-compact polyhedra, and the dimension of the target Minkowski space always remains optimal.

\subsection{Preliminaries and notation}

We denote an indefinite metric polyhedron by a triple $(\Chi, \Tau, g)$, where $(\Chi,\Tau)$ is the simplicial complex, and $g$ is the function $g:E(\Tau)\to\R$ associating a squared length to every edge. Here $E(\Tau)$ is the set of edges of $\Tau$. 


A \textit{Minkowski space} of signature $(p,q)$ denoted by $\R^p_q$ is a vector space $\R^{p+q}=\{v=(v^+,v^-)|v^+\in \R^p,v^-\in\R^q\}$ endowed with the following pseudoscalar product: 
$$\< v,u\>_{\R_q^p}:=\< v^+,u^+\>_{\R^p}-\< v^-,u^-\>_{\R^q},$$
where $\< \cdot,\cdot\>_{\R^d}$ denotes the standard scalar product in $\R^p$ or $\R^q$. We restrict our attention to Minkowski spaces of signature $(d,d)$ for some $d$.

Let $Q\subset\R^d_d$ be a collection of points. We say that the points of $Q$ are in $d$-general position if and only if any subcollection of at most $d+1$ points of $Q$ forms an affinely independent set.


Let $(\Chi,\Tau,g)$ be an indefinite metric polyhedron with vertex set $\Nu$ 
and let $\Nu'\subset \Nu$ be an arbitrary subset of $\Nu$. A map 
$m:\Nu'\to\R^d_d$ is called a \textit{partial simplicial isometric map} if for 
every simplex $S$ of $\Chi$ whose vertices belong to $\Nu'$, $m$ preserves the 
bilinear form of $S$ . Equivalently, for every edge $l$ with endpoints 
$a,b\in\Nu'$, the squared length of $l$ equals $\<m(a)-m(b),m(a)-m(b)\>_{\R_q^p}$.

For a positive integer $k$, a finite or countable graph is called \textit{$k$-degenerate} if every its subgraph has a vertex of degree at most $k$.  Equivalently speaking, a graph is $k$-degenerate if and only if its vertices can be ordered so that each vertex has at most $k$ neighbors that are earlier in the ordering. Clearly, a graph with maximal degree $d$ is $d$-degenerate.

Similarly, we say that an indefinite metric polyhedron is \textit{$d$-degenerate} if its $1$-dimensional skeleton is a $d$-degenerate graph.

\subsection{Main results}

All the indefinite metric polyhedra are assumed to be countable or finite.

We begin with a constructive version of theorems from \cite{M} for non-compact indefinite metric polyhedra:
\begin{theorem}
\label{thm:cor}
Let $(\Chi,\Tau,g)$ be a (possibly non-compact) $n$-dimensional $d$-degenerate indefinite metric polyhedron. Then
\begin{enumerate}
  \item There exists a simplicial isometric map of $\Chi$ into $\R^d_d$.
  \item In addition, assume that $d\geq 2n+1$. Then there exists a simplicial 
isometric embedding of $\Chi$ into $\R^d_d$.
\end{enumerate}
\end{theorem}
In Section \ref{sect:algo} we provide a simple algorithm to construct such maps.

Our main theorem states that every partial simplicial isometric map of an
indefinite metric polyhedron with uniformly bounded vertex degrees extends to a simplicial isometric map of the whole 
polyhedron.

\begin{theorem}
\label{thm:main}
  Let $G=(\Chi,\Tau,g)$ be an indefinite $n$-dimensional metric polyhedron 
with vertex set $\Nu$, and let $d:= \max\{\deg(v)|v \in \Nu\}$.
  \begin{enumerate}
    \item Let $\tau':\Nu'\to\R^d_d$ be a partial simplicial isometric map of 
$G$ into a Minkowski space such that the  points $\tau'(\Nu')$ are in 
$d$-general position. Then there exists a simplicial isometric map 
$\tau:\Nu\to\R^d_d$ of $G$ such that $\tau|_{\Nu'}=\tau'$ and the points 
$\tau(\Nu)$ are also in $d$-general position.
    \item If in addition $d\geq 2n+1$, then both $\tau$ and $\tau'$ are 
automatically injective, so $\tau$ is a simplicial isometric embedding.
  \end{enumerate}
 \end{theorem}

Theorem \ref{thm:cor} is a special case of Theorem \ref{thm:main} with 
$\Nu'=\emptyset$, but we provide a separate short proof for this important case.


\section*{Acknowledgments}
We are very grateful to our adviser Sergei V. Ivanov for formulating the problem and for all his ideas and advice. We would like to express our gratitude to Barry Minemyer
for his valuable comments on the first version of the paper.


\section{Proof of Theorem \ref{thm:cor}}

A linear subspace $H$ of $\R^d_d$ is called \textit{isotropic} if any vector $h\in H$ has zero squared length, i.e. $\< h,h\>=0$.

Fix two arbitrary complementary isotropic $d$-dimensional subspaces $\Sigma,\Delta$ of $\R^d_d$. This can be easily achieved, for example, by putting
$$\Sigma:=\{(v^+,v^-)|v^+-v^-=0\}$$
and
$$\Delta:=\{(v^+,v^-)|v^++v^-=0\}.$$
We denote by $P_\Sigma$ and $P_\Delta$ the projection operators on these subspaces with respect to the direct sum decomposition $\R^d_d = \Sigma\oplus\Delta$. That is, for any $v\in\R^d_d$,  $P_\Sigma(v)\in\Sigma$, $P_\Delta(v)\in\Delta$ and $P_\Sigma(v)+P_\Delta(v)=v$.

The following lemma is used in the proofs of Theorems \ref{thm:cor} and \ref{thm:main}:
\begin{lemma}
\label{lemma:linear}
Let $\R^d_d = \Sigma\oplus\Delta$, where $\Sigma,\Delta$ are two complementary isotropic $d$-dimensional subspaces  of $\R^d_d$. Let $v_0$ be a point in $\Delta$, let $k\le d$ be a positive integer, and let $u_1,\dots,u_k$ be a set of points  in $\R^d_d$ such that $v_0$, $P_{\Delta}(u_1),\dots,P_{\Delta}(u_k)$ are affinely independent  in $\Delta$. Then for any sequence $c_1,\dots,c_{k}$ of real numbers there is a point $u_0\in \R^d_d$ such that

\begin{enumerate}
  \item for every $1\leq i\leq k$, $\< u_0-u_i,u_0-u_i\>_{\R^d_d}=c_i$,
  \item $P_\Delta(u_0)=v_0.$
\end{enumerate}
\end{lemma}
\begin{proof}

It is easy to see that for any vector $v\in\R^d_d$, one has
$$\< v,v\>_{\R^d_d}=2\< P_\Delta(v),P_\Sigma(v)\>_{\R^d_d},$$
because the subspaces $\Delta$ and $\Sigma$ are required to be isotropic. Note that $\< \cdot,\cdot \>_{\R^d_d}|_{\Delta \times \Sigma}$ is a non-degenerate bilinear pairing between $\Delta$ and $\Sigma$.

System (1) is equivalent to the system 
$$2\<P_\Delta(u_0-u_i),P_\Sigma(u_0-u_i)\>_{\R^d_d}=c_i, 1\leq i\leq k.$$
Denote $P_\Delta(u_i)$ by $v_i$ and $P_\Sigma(u_i)$ by $h_i$ for $1\leq i\leq k$. We construct $u_0$ in the form $u_0 = v_0 + h_0$ where $h_0 \in \Sigma$, so the system now takes the form

\begin{equation}
2\<v_0-v_i,h_0-h_i\>_{\R^d_d}=c_i, i=1..k.
\label{ms}
\end{equation}

 We need to find $h_0$, and all the other vectors are fixed.

 Clearly, this is a linear equation system, which is non-degenerate, because the vectors $v_i$ are affinely independent and ${\<\cdot,\cdot\>_{\R^d_d}|}_{\Delta \times \Sigma}$ is a non-degenerate pairing, so the required vector $h_0$ (and, therefore, $u_0$) exists. \end{proof}
%
%
%
%
%

\subsection{The algorithm}
Fix two arbitrary complementary isotropic $d$-dimensional subspaces 
$\Sigma,\Delta$ of $\R^d_d$, and fix an infinite sequence of points 
$v_0,v_1,\dots \in \Delta$ in $d$-general position (for example, these points 
can lie on the Moment Curve). Let  $G=(\Chi, \Tau, g)$ be a $d$-degenerate $n$-dimensional 
indefinite metric polyhedron with vertex set 
$\Nu=\{t_0,t_1,\dots\}$, where each vertex $t_i$ is connected to at most $d$ vertices from the set $\{t_0,\dots,t_{i-1}\}$. We need to find a simplicial isometric map 
$\tau:\Chi\to\R^d_d$ of $G$ into $\R^d_d$.
Since any simplicial map is completely determined by images of vertices, it is enough to define $\tau$ on the vertex set $\Nu$.
First we choose $\tau(t_0)$, then 
$\tau(t_1)$, and so on.

For every $i\geq0$, there are at most 
$d$ points in $\{t_j\}_{j< i}$ connected to $t_i$ by an edge. If this set 
is empty, then put $\tau(t_i)$ to be any point of $\R^d_d$ such that 
$P_\Delta(\tau(t_i))=v_i$. If this set is non-empty, apply Lemma 
\ref{lemma:linear} to its points by solving the system~\eqref{ms} and obtain 
$\tau(t_i)$ such that $P_\Delta(\tau(t_i))=v_i$.
\label{sect:algo}

This completes the first part of Theorem \ref{thm:cor}. Now the second 
part of Theorem \ref{thm:cor} is a consequence of the following 
lemma from \cite{M}:
\begin{lemma}
\label{lemma:emb}
  Let $(\Chi, \Tau, g)$ be an $n$-dimensional metric polyhedron, let $f:\Chi\to 
\R^N$ be a simplicial map with respect to $\Tau$, and let $\Nu$ be the
vertex set of $\Chi$. If $f(\Nu)$ is in $(2n + 1)$-general position then $f$ is 
an embedding.
\end{lemma}
\begin{proof}
  Indeed, if the images of two non-intersecting simplices intersect, then their 
vertices cannot be in general position. \end{proof}
If the points $P_\Delta(\tau(t_i))$ lie in $d$-general position then 
the points $\tau(t_i)$ also lie in $d$-general position, so if $d\geq 2n+1$ 
then this lemma implies the desired result. \qed

\section{Proof of Theorem \ref{thm:main}}

\begin{proof}
We assume that $\Nu'$,$\Nu$,$d$,$g$,$\tau'$ are the same as in the 
statement of Theorem \ref{thm:main}. Since our extension can be performed one 
point at a time, it suffices to consider the case $\Nu=\Nu'\cup\{t_0\}$. Let 
$H:=\{v_1,\dots,v_m\}$ be the set of images of vertices adjacent to $t_0$, so $m \le d$. 
By the assumption of Theorem \ref{thm:main}, $v_1,\dots,v_m$ are affinely independent. Let 
$\Sigma$ and $\Delta$ be two fixed complementary isotropic $d$-dimensional 
subspaces of $\R^d_d$, for example, we can again put 
$$\Sigma:=\{(v^+,v^-)|v^+-v^-=0\}$$
and
$$\Delta:=\{(v^+,v^-)|v^++v^-=0\}.$$
If the points 
$P_\Delta(v_1),\dots,P_\Delta(v_m)$ are also affinely independent, 
then we can take almost any $v_0 \in \Delta$ and apply Lemma \ref{lemma:linear} 
to $H$ to obtain the point $\tau(t_0)$ with the required properties. 
Unfortunately, this is not always the case: the points 
$P_\Delta(v_1),\dots,P_\Delta(v_m)$ can easily be affinely dependent even if 
the points $v_1,\dots,v_m$ are not.

But instead of having a fixed pair of $\Sigma$ and $\Delta$, for every new 
vertex we can choose its own pair of isotropic subspaces depending on the set of neighbors of that vertex. And this additional freedom actually allows us to 
apply Lemma \ref{lemma:linear}, because of the following observation:
%

\begin{lemma}
\label{lemma:aut}
  Let $H$ be a set of at most $d+1$ affinely independent points in $\R^d_d$. Then there exist two complementary isotropic $d$-dimensional subspaces $\Sigma_H,\Delta_H$ such that the points of $P_{\Delta_H}(H)$ are also affinely independent.
\end{lemma}
\begin{proof}
The case $|H|<d+1$ clearly follows from the case $|H|=d+1$, so we assume $|H|=d+1$. Without loss of generality we may assume that $H$ contains the origin.

\def\hh{{U}}
Let $\hh$ denote the linear subspace of $\R^d_d$ spanned by the points of $H$. Let $P_+:\R^d_d \to \R^d_d$ denote a projection operator on the positive coordinate component. We may assume that $P_+(U)$ is the coordinate subspace spanned by the vectors $e_1,\dots,e_k$, where  the vectors $e_1,\dots,e_k$ are the first $k$ vectors of the standard basis of the positive coordinate component, otherwise standard basis of the positive coordinate component should be substituted by another one which satisfies this condition. It is possible to choose a basis $u_1,\dots,u_d$ of $U$ such that $P_+(u_i) = e_i$ for $i=1..k$, and $P_+(u_i) = 0$ for $i=k+1..d$.

Let
$$\Sigma:=\{(v^+,v^-)|v^+-v^-=0\}$$
$$\Delta:=\{(v^+,v^-)|v^++v^-=0\}.$$

We are going to find an isometry $f$ of $\R^d_d$, such that the subspaces $\Sigma_H:=f^{-1}\Sigma$ and $\Delta_H:=f^{-1}\Delta$ meet the requirements of the lemma, namely, $f$ has to be such that the points $P_{f^{-1}\Delta}(H)$ are affinely independent, or, equivalently, $P_\Delta(f\hh)=\Delta$.


\def\hh{{U}}


%



Now we explicitly describe the matrix of the desired Lorentz transformation $f$. Let $M^+$ and $M^-$ denote positive and negative components of $\R^d_d$. We are going to find $f$ in the form $f=(f^+,f^-)$, where $f^+:M^+\to M^+$ and $f^-:M^-\to M^-$ are two linear isometries. We put $f^+:=\textrm{id}$. Now we are going to find the matrix $F^-$ of $f^-$ . Let $S$ be the $2d \times d$ matrix whose columns are coordinates of $u_i$, it can be split into two $d \times d$ matrices $S^+$ and $S^-$. The subspace $P_\Delta(f(\hh))$ has dimension $d$ if and only if

$$\det(S^+ - F^-S^-) \ne 0,$$ 
because for any vector $v=(v^+,v^-)$ one has $P_\Delta(v)=\frac{1}{2}(v^+-v^-,v^--v^+)$.

By construction, the matrix $S^+$ has the first $k$ diagonal entries equal to $1$, while all the other entries of $S^+$ are zeroes. Hence the columns of $S^-$ with numbers $k+1, \dots, d$ have to be linearly independent, because $\operatorname{dim} U = d$.

Now we need to make the following standard decomposition of the matrix $S^-$ called \textit{QL-decomposition}, see, for example, \cite[\S5.2]{Golub}. Namely, if $A$ is a $d\times d$-matrix, then there exist an orthogonal matrix $Q$ and a lower-triangular matrix $L$ with non-negative diagonal entries, such that $A=QL$. The proof is a slight modification of the Gram–Schmidt orthogonalization process.  Note that if the last $d-k$ columns of $A$ were linearly independent, then the last $d-k$ diagonal entries of $L$ become strictly positive, because the columns remain linearly independent after the multiplication by $Q$ from the left.

So, let $-S^-=Q_-L_-$ be such QL-decomposition of $-S_-$. Put $F^-:=Q_-^T$, where by $^T$ we denote the matrix transpose. Then $S^+-F^-S^-=S_++L_-$. The matrix on the right-hand side is clearly lower-triangular and has strictly positive diagonal entries, so its determinant is also strictly positive, which concludes the proof of the lemma.\end{proof}



Now we have the following situation: we are given a set $H:=\{u_1,\dots,u_m\}$ of points in 
$\R^d_d$ which are the images of the vertices $t_1,\dots,t_m$ adjacent to $t_0$, so $m \le d$. Then Lemma 
\ref{lemma:aut} provides us with a pair $\Sigma_H,\Delta_H$ of complementary 
isotropic $d$-dimensional subspaces of 
$\R^d_d$ such that the points of $P_{\Delta_H}(H)$ are affinely independent. We 
want to choose a point $v_0$ in such a way that 
the points of $\{v_0\}\cup P_{\Delta_H}(H)$ are affinely independent. If we 
manage to do that, then we can apply Lemma \ref{lemma:linear}, taking each $c_i$ to be the squared length of 
the edge connecting $t_0$ and $t_i$. Lemma \ref{lemma:linear} returns a 
point $u_0$ such that $P_{\Delta_H}(u_0)=v_0$ and if we put $\tau(t_0):=u_0$, 
then $\tau$ will preserve the squared length of every edge. This almost 
completes the proof, but there is also a requirement in the statement of the 
theorem that the points $\tau(\Nu'\cup\{t_0\})$ must be in $d$-general 
position, so we need to be slightly more careful when we choose $v_0$. Namely, 
if for every $d$-tuple of vertices $t_1,\dots,t_d 
\in \Nu$ the point $v_0$ does not lie in the $P_{\Delta_H}$-image of the affine 
span of $\tau(t_1),\dots,\tau(t_d)$, then the points $\tau(\Nu'\cup\{t_0\})$ 
are in $d$-general position. But this forbids $v_0$ to lie in a union of only a 
countable number 
of hyperplanes in $\Delta_H$, therefore such a point $v_0$ exists. Thus, all 
one 
needs to do now is to apply Lemma \ref{lemma:linear} to the set $H$ and extend $\tau$ to the 
vertex $t_0$. \end{proof}

\bibliography{pseudo_euclid}
\bibliographystyle{plain}

\end{document}